\documentclass[11pt]{amsart}

\usepackage[paper=letterpaper,margin=1in,twoside=false,includehead,headheight=18pt]{geometry}

\usepackage[utf8]{inputenc}

\usepackage{amsmath,amssymb,amsthm} 
\usepackage{enumerate,paralist}
\usepackage{graphicx}
\usepackage{mathtools}
\usepackage{todonotes}

\newtheorem{theorem}{Theorem}[section]

\newtheorem{corollary}[theorem]{Corollary}

{\theoremstyle{definition} 

 }

\newtheoremstyle{named}%
  {}{}						
  {\upshape}				
  {0pt}{\bfseries}			
  {.}						
  {.5em}					
  {\thmname{#1}\thmnote{ #3}}  

\theoremstyle{named}

\newcommand{\set}[1]{\left\{{#1}\right\}} 
\newcommand{\setof}[2]{\left\{{#1}\,:\,{#2}\right\}}
\newcommand{\of}{\subseteq}

\DeclareMathOperator{\Ind}{Ind}
\DeclarePairedDelimiter{\abs}{\lvert}{\rvert}

\AtBeginDocument{%
   \def\MR#1{}
}

\begin{document}

\title{A note on the alternating number of independent sets in a graph}
\date{\today}
\author{Jonathan Cutler} \address{Department of Mathematics\\
Montclair State University\\
Montclair, NJ} \email{jonathan.cutler@montclair.edu} 
\author{Nathan Kahl} \address{Department of Mathematics and Computer Science\\
Seton Hall University\\
South Orange, NJ} \email{nathan.kahl@shu.edu}
\author{Phoebe Zielonka}
\address{Department of Mathematics\\
Montclair State University\\
Montclair, NJ} \email{zielonkap1@montclair.edu} 

\begin{abstract}
  The independence polynomial of a graph $G$ evaluated at $-1$, denoted here as $I(G;-1)$, has arisen in a variety of different areas of mathematics and theoretical physics as an object of interest.  Engstr\"om  used discrete Morse theory to prove that $\abs{I(G;-1)}\leq 2^{\phi(G)}$ where $\phi(G)$ is the decycling number of $G$, i.e., the minimum number of vertices needed to be deleted from $G$ so that the remaining graph is acyclic.  
	Here, we improve Engstr\"om's bound by showing $\abs{I(G;-1)}\leq 2^{\phi_3(G)}$ where $\phi_3(G)$ is the minimum number of vertices needed to be deleted from $G$ so that the resulting graph contains no induced cycles whose length is divisible by $3$.  We also note that this bound is not just sharp but that every value in the range given by the bound is attainable by some connected graph.
\end{abstract}
\maketitle

\section{Introduction}\label{sec:intro}

Let $\alpha(G)$ denote the \emph{independence number of a graph $G$}, the maximum order of an independent set of vertices in $G$.  
The \emph{independence polynomial} $I(G;x)$ of a graph $G$ is the generating function for the number of independent sets in $G$. In other words,
\[
  I(G;x)=\sum_{k=0}^{\alpha(G)} i_k(G)x^k,
\]
where $i_k(G)$ is the number of independent sets of size $k$ in $G$, so, for example, we have that $I(G;1)$ is the total number of independent sets in $G$.

The independence polynomial has been shown to have important connections with some models in theoretical physics, most prominently with hard-core models of fermions on a lattice.  In particular, in this setting $I(L;-1)$, the independence polynomial of a lattice $L$ evaluated at $-1$, corresponds (up to sign) to the Witten index of that lattice from statistical mechanics.  The quantity $I(G;-1)$ also represents important quantities in algebraic topology.  Letting $\Ind(G)$ denote the independence complex of a graph $G$, then $I(G;-1)$ gives the reduced Euler characteristic of $\Ind(G)$.  A significant interdisciplinary literature has grown up taking advantage of these connections in order to calculate the Witten index, homology groups or indeed the complete homotopy type of the independence complex for various lattices, see, e.g., \cite{a13,bmln08,fs05,fsv05,hhfs08,hs10,j06,j09} for some examples.

Returning to graph theory, the quantity $I(G;-1)$ represents the number of independent sets of $G$ of even cardinality minus the number of independent sets with odd cardinality, a quantity sometimes called the \emph{alternating number of independent sets} in a graph.  Call a cycle whose length is divisible by three a \emph{$\tilde{3}$-cycle}.  Motivated by topological and graph coloring considerations, in the early 1990's Kalai and Meshulam conjectured that the presence of induced $\tilde{3}$-cycles were necessary to produce large magnitudes of $I(G;-1)$.  Specifically, they conjectured that in any \emph{ternary graph}, that is, any graph without induced $\tilde{3}$-cycles, the magnitude of the alternating number of independent sets could be at most 1.  This conjecture was eventually proved by Chudnovsky, Scott, Seymour and Spirkl.
\begin{theorem}[Chudnovsky et al.~\cite{CSSS}]\label{thm:csss}
    If $G$ is a ternary graph, then $\abs{I(G;-1)}\leq 1$.
\end{theorem}
\noindent Various generalizations of this result have since appeared, see, e.g., \cite{E2,Kim}. 

Prior to the proof of the Kalai-Meshulam conjecture, Engstr\"om~\cite{E} had proved a bound on the magnitude of $I(G;-1)$ for any graph $G$ which involved decycling sets.   A \emph{decycling set} of a graph $G$ is a set of vertices $S \subseteq V(G)$ such that $G-S$ is
acyclic.  We denote by $\Phi(G)$ the set of all decycling sets of $G$, and define the \emph{decycling number of a graph $G$}, denoted $\phi(G)$, to be the minimum order of a decycling set of $G$.  With this notation, Engstr\"om's result is as follows.
\begin{theorem}[Engstr\"om \cite{E}]\label{thm:e}
  If $G$ is any graph, then $\abs{I(G;-1)}\leq 2^{\phi(G)}$.
\end{theorem}
While Engstr\"om's original proof of Theorem \ref{thm:e} used discrete Morse theory, Levit and Mandrescu \cite{LM} later gave an elementary inductive proof using a recursive formula for $I(G;x)$.  Levit and Mandrescu also conjectured that every value in the allowable range given by Theorem~\ref{thm:e} is achievable by some connected graph, which was proved by the first two authors.
\begin{theorem}[Cutler, Kahl~\cite{CK}]\label{thm:ck}
  Given a positive integer $k$ and an integer $q$ with  $\abs{q}\leq2^k$, there is a connected graph $G$ with $\phi(G)=k$ and $I(G;-1) = q$.
\end{theorem}

Given the importance of $\tilde{3}$-cycles to the evaluation of the independence polynomial at $-1$ demonstrated by Theorem \ref{thm:csss}, it seems natural to ask if Engstr\"om's bound in Theorem \ref{thm:e} might be tightened to reflect this.  Recall that the \emph{cyclomatic number of a graph $G$}, denoted $\nu(G)$, is the minimum number of edges $S \subseteq E(G)$ such that $G - S$ is acyclic.  It is well-known that the cyclomatic number is  $\nu(G)=e(G)-n(G)+q(G)$, where $q(G)$ is the number of components in $G$.   In this direction, recently Cao and Ren~\cite{CR} proved the following.
\begin{theorem}[Cao, Ren~\cite{CR}] \label{thm:cr}
  If $G$ is a graph that contains a cycle whose length is not divisible by $3$, then 
  \[
    \abs{I(G;-1)}\leq 2^{\nu(G)}-\nu(G).
  \]
\end{theorem}
\noindent Theorem \ref{thm:cr} actually represents an improvement on an earlier bound of 
Levit and Mandrescu~\cite{LM2} (which omitted the last term), and does represent an improvement of Engstr\"om's bound for those graphs $G$ for which $\phi(G)=\nu(G)$.  However, Theorem \ref{thm:cr} explicitly does not pertain to graphs with $\tilde{3}$-cycles, and even in graphs with $\tilde{3}$-cycles it is easy to see that in general $\phi(G) \le \nu(G)$, and the difference between $\phi(G)$ and $\nu(G)$ can be arbitrarily large.  Hence the question still remains: what is a natural, best possible bound for $I(G;-1)$ on general graphs that takes $\tilde{3}$-cycles into account?  In this short note, we answer this question.

Call a set $D \subseteq V(G)$ a \emph{$\tilde{3}$-decycling set of $G$} if $G-D$ is ternary, i.e., contains no induced $\tilde{3}$-cycles.  Let the \emph{ternary decycling number of a graph $G$}, denoted $\phi_3(G)$, be the minimum size of a $\tilde{3}$-decycling set of $G$, and let $\Phi_3(G)$ denote the set of all such decycling sets.    We prove the following.

\begin{theorem}\label{thm:main}
  If $G$ is any graph, then
  \[
    \abs{I(G;-1)}\leq 2^{\phi_3(G)}.
  \]
\end{theorem}

The bound above is best possible.  In fact we can show more, demonstrating that a theorem analogous to Theorem \ref{thm:ck} holds for this new bound as well.

\section{Results}\label{sec:results}

The proof of our result follows similar lines as Levit and Mandrescu's proof of Engstr\"om's original bound.  We will actually prove the following chain of inequalities.  For a set $D\of V(G)$, we let $G[D]$ be the subgraph of $G$ induced by the vertices of $D$.

\begin{theorem}
  If $G$ is any graph, then
  \[
    \abs{I(G;-1)}\leq \min_{D \in \Phi_3(G)} \abs{\Ind(G[D])} \leq 2^{\phi_3(G)}.
  \]
\end{theorem}

\begin{proof}
We will require the following recursive formula for $I(G;x)$ (see, e.g., \cite{gh83,hl94}), namely that 
    \[
      I(G;x)=I(G-v;x)+xI(G-N[v];x),
    \]
where $N[v]=\setof{u\in V(G)}{uv\in E(G)}\cup \set{v}$ is the closed neighborhood of $v$.  Thus, setting $x=-1$, we have
    \[{}
      I(G;-1)=I(G-v;-1)-I(G-N[v];-1).
    \] 
 
To prove the first inequality of the bound, we show that $\abs{I(G;-1)} \le \abs{\Ind(G[D])}$ for any $D \in \Phi_3(G)$ by induction on the order of $D$.  For the base case, if $|D|=0$ then $G$ is ternary and so by Theorem~\ref{thm:csss} we have $\abs{I(G;-1)} \le 1$.  Thus it suffices to show that for $G$ ternary there is no $D \in \Phi_3(G)$ such that $\Ind(G[D])$ is empty.   But the null set $\emptyset \in \Ind(G[D])$ for any $D \in \Phi(G)$, and hence $\abs{\Ind(G[D])} \ge 1$.   

Now assume $k\geq 0$ and that the first inequality is true for all $D \in \Phi_3(G)$ with $|D|\leq k$.  If $G$ is a graph with $D \in \Phi(G)$ such that $|D|=k+1$, then there is a vertex $v\in V(G)$ such that $D-v$ is a $3$-decycling set of $G-v$ and $D-N[v]$ is a $3$-decycling set of $G-N[v]$, and clearly $|D-v| \le k$ and $|D-N[v]| \le k$.  Thus,
    \begin{align*}
      \abs{I(G;-1)}&= \abs{I(G-v;-1)-I(G-N[v];-1)}\\
      &\leq \abs{I(G-v;-1)}+\abs{I(G-N[v];-1)}\\
      &\leq \abs{\Ind(G-v)[D]}+\abs{\Ind((G-N[v])[D]}\\
      &=\abs{\Ind(G[D])}.
    \end{align*}
This proves the first inequality.  The second inequality now follows by noting that $\abs{\Ind(G[D])} \le 2^{|D|}$ for any $D \subseteq V(G)$.  Letting $D$ be any minimum decycling set produces the second inequality, and completes the proof.  \qedhere
\end{proof}

The bound $\abs{I(G;-1)} \le 2^{\phi_3(G)}$ is not only tight, but a ``density'' result analogous to Theorem~\ref{thm:ck} is true for it as well.  In fact, since clearly $\phi_3(G)\leq \phi(G)$, the result below follows directly from  the proof of Theorem~\ref{thm:ck}.  It is easily seen there that in the constructions given all cycles are disjoint and all cycles have length divisible by $3$ and thus for those graphs we have $\phi(G)=\phi_3(G)$.  The following is an immediate consequence.

\begin{corollary}
    Given a positive integer $k$ and an integer $q$ with  $\abs{q}\leq 2^k$, there is a connected graph $G$ with $\phi_3(G)=k$ and $I(G;-1)=q$. 
\end{corollary}

\bibliographystyle{amsplain}
\bibliography{ternary}

\end{document}